\documentclass[12pt,a4paper]{amsart}
\usepackage{amsthm, amssymb}
\usepackage{amsfonts}
\usepackage{amsthm}
\usepackage{amsmath}
\usepackage{amscd}
\usepackage[latin2]{inputenc}
\usepackage{t1enc}
\usepackage[mathscr]{eucal}

\numberwithin{equation}{section}
\usepackage[margin=2.9cm]{geometry}
\usepackage{epstopdf} 
\usepackage{hyperref}
\newtheorem{proposition}{Proposition}[section]

 \newtheorem{corollary}{Corollary}[section]
 \newtheorem{definition}{Definition}[section]
\newtheorem{remark}{Remark}[section]
\theoremstyle{definition}
\newtheorem{theorem}{Theorem}[section]
\theoremstyle{plain}

\newtheorem{lemma}{Lemma}[section]

 \theoremstyle{definition}

\title[]{ A class of irreducible modules for loop-Virasoro algebras}

\author[]{Priyanshu Chakraborty AND Punita Batra }

\address{Priyanshu Chakraborty, Harish-Chandra Research Institute (HBNI), Chhatnag Road, Jhunsi, Prayagraj(Allahabad) 211019, Uttar Pradesh, India.}
 
\email{priyanshuchakraborty@hri.res.in}

\address{Punita Batra, Harish-Chandra Research Institute (HBNI), Chhatnag Road, Jhunsi, Prayagraj(Allahabad) 211019, Uttar Pradesh, India.}

\email{batra@hri.res.in}

\begin{document}
\begin{abstract}
Tensor product of highest weight modules and intermediate modules for Virasoro algebra have been studied around 1997. Since then the irreducibility problem for tensor product of modules is open. We consider the loop-Virasoro algebra $Vir \otimes B$, where $Vir$ is the Virasoro algebra and  $B$ a commutative associative unital algebra over $\mathbb C$. In this paper we study the irreducibility problem for the tensor product of highest weight modules and intermediate modules for $Vir\otimes B$. Finally we find out a necessary and sufficient conditions for such modules to be isomorphic. 
\end{abstract}

\maketitle
Keywords: {Virasoro algebra, Verma module, Intermediate module}\\

MSC 2020: {17B65,17B66,17B68}

\section{Introduction}
Importance of the representation theory of infinite dimensional Lie algebras are well known in both mathematics and physics. Many mathematicians studied representations for the Lie algebra of the derivation algebra of $A=\mathbb C[t_1^{\pm 1}, t_2^{\pm 1},...,t_n^{\pm 1}]$, so called Witt algebra denoted by $W_n$ or $DerA_n$, for references \cite{16,17,18}. In particular, for $n=1$ the universal central extension for the derivation algebra of $A$ is known as the Virasoro algebra, generally denoted by $Vir$. Representations for  Virasoro algebra have been well studied by several mathematicians, for references \cite{2,3,4}. In \cite{4}, O.Mathieu proved that a simple module for $Vir$ with finite dimensional weight spaces is either a highest weight module, a lowest weight module or a uniformly bounded module of intermediate series. Later in \cite{2}, it was proved that an irreducible module with an infinite dimensional weight space should have all infinite dimensional weight space. From then a large class of irreducible modules for $Vir$ with infinite dimensional weight spaces have been found. Simple modules with infinite dimensional weight spaces for $Vir$ were constructed in \cite{5} for the first time, by taking tensor product of highest weight modules and intermediate modules. Later isomorphism class of above mentioned modules have been found in \cite{6}. Another large class of irreducible modules for $Vir$ with infinite dimensional weight spaces have been found in \cite{7} using non-weights modules introduced in \cite{8}. \\

On the other hand, in present days mathematicians are showing interest to study representations of loop algebras of some well known  Lie algebras. For instance, representations of loop algebras of well known Lie algebras have been studied in, \cite{9,10,11,12,13,14,15}. Representations of loop of Virasoro algebras have been studied in \cite{9,10}.  For $B=\mathbb C[t^{\pm 1}]$, complete classification of irreducible $Vir\otimes B$-modules have been done in \cite{9}. Later this work was generalized by A.Savage for $Vir \otimes B$, where $B$ is a commutative associative finitely generated unital algebra over $\mathbb C$ in \cite{10}. In \cite{10}, A.Savage proved that a simple module for $Vir\otimes B$ with finite dimensional weight spaces is either highest weight module, lowest weight module or uniformly bounded module. But classification of irreducible $Vir\otimes B$ modules with a finite dimensional weight space is still open. In this paper we have constructed a class of irreducible modules for $Vir \otimes B$, for a commutative associative unital algebra $B$ over $\mathbb C$ with infinite dimensional weight spaces. \\

This paper has been organised as follows. In section 2, we start with basis definitions and preliminaries. We define Verma module $M(\phi)$ for $Vir\otimes B$ and define a $Vir\otimes B$ module structure on intermediate modules for $Vir$, these modules are denoted by $V_{\alpha,\beta,\psi},$ where $\alpha, \beta \in \mathbb C$ and $\psi:B\to \mathbb C$ be an algebra morphism with $\psi(1)=1$. Then we consider the tensor product of irreducible quotient $V(\phi)$ of $M(\phi)$ and the irreducible modules obtained from the family $V_{\alpha,\beta,\psi},$ denote these irreducible modules as $V'_{\alpha,\beta,\psi}$. We denote $V^\phi_{\alpha,\beta,\psi}=V(\phi)\otimes V'_{\alpha,\beta,\psi},$ and study the properties of these modules in section 3. Finally in Section 4, we find isomorphism classes of these modules.

\section{Notations and Preliminaries}
\subsection*{(1)} Throughout the paper all the vector spaces, algebras, tensor products are taken over the field of complex numbers $\mathbb{C}$. Let $\mathbb{Z}$, $\mathbb{N}$, $\mathbb{Z}^+$ denote the sets of integers, natural numbers and non-negative integers respectively. For any Lie algebra $G$, let $ U (G)$ denote the universal enveloping algebra of $G$.
\subsection*{(2)} The Virasoro algebra is an infinite dimensional Lie algebra with the basis $\{ d_n,C:n \in \mathbb Z \}$ and defining relations
\begin{align}\label{a2.1}
[d_m,d_n]=(n-m)d_{m+n}+\delta_{m,-n}\frac{m^3-m}{12}C,\\
[d_n,C]=0,
\end{align}
for all $n,m \in \mathbb Z$.\\
We denote the Virasoro algebra as $Vir$. Clearly $Vir$ has a triangular decomposition as, $$Vir= Vir^- \oplus Vir^0 \oplus Vir^+, $$ 

where $Vir^-=\displaystyle{ \bigoplus_{n<0}}\mathbb C d_n$, $Vir^0= span\{d_0,C\}$ and $Vir^+=\displaystyle{ \bigoplus_{n>0}}\mathbb C d_n$.
\subsection*{(3)} It is well known from \cite{1}, that there is a class of intermediate modules $V_{\alpha, \beta}$ for $Vir$ with two parameters $\alpha, \beta \in \mathbb C$. As a vector space $V_{\alpha, \beta}=\displaystyle{ \bigoplus_{n \in \mathbb Z}}\mathbb C v_n$ and $Vir$ action on $V_{\alpha, \beta}$ is given by,
\begin{align}\label{a2.3}
d_n.v_k=(\alpha+k+n\beta)v_{k+n},\\
C.v_k=0,
\end{align}
for all $n,k \in \mathbb Z$. \\
For convenience we denote $V'_{\alpha, \beta}= V_{\alpha,\beta}$, if $V_{\alpha, \beta}$ is irreducible.
\begin{lemma}\label{l2.1}$($\cite{1}$)$
\item[1.] The $Vir$ module $V_{\alpha, \beta} \simeq V_{\alpha+m, \beta}$ for all $m \in \mathbb Z$.\\

\item[2.] The $Vir$ module $V_{\alpha, \beta}$ is irreducible if and only if $\alpha \notin \mathbb Z $  and $\beta \notin \{0,1\}$.\\
\item[3.] $V_{0,0}$ has a unique non-trivial proper sub-module $\mathbb Cv_0$ and denote $V_{0,0}'=V_{0,0}/\mathbb Cv_0$.\\
\item[4.] $V_{0,1}$ has a unique non-zero proper sub-module $V'_{0,1}=\displaystyle{\bigoplus_{i\neq 0}} \mathbb C v_i$. \\
\item[5.]  $V'_{\alpha,0}\simeq V'_{\alpha,1}$ for all $\alpha \in \mathbb C$.

\end{lemma}
\begin{remark}
To avoid repetition through out this paper we always take $V'_{\alpha, \beta}$ for $0\leq Re \alpha <1, \beta \neq 1$, where $Re \alpha$ is the real part of $\alpha$. Clearly this class of modules are irreducible for $Vir$.
\end{remark}
\subsection*{(3)} Let $B$ be a commutative associative unital algebra. We consider the Lie algebra $Vir_B= Vir\otimes B$ with the following bracket operation, 
  $$[X\otimes b, Y \otimes b']=[X,Y]\otimes bb' $$
  for all $b,b' \in B$ and $X,Y \in Vir$. We denote $X \otimes 1$ simply as $X$ for all $X \in Vir$. Note that $Vir_B$ has a triangular decomposition as,
   $$Vir_B=Vir^-_B \oplus Vir^0_B \oplus Vir^+_B,$$
    where $Vir^-_B=Vir^-\otimes B, Vir^0_B=Vir^0\otimes B$ and $Vir^+_B=Vir^+\otimes B$.
\subsection*{(4)} We define a $Vir_B$ module structure on $V_{\alpha, \beta}$. Let $\psi:B \to \mathbb C$ be an algebra homomorphism such that $\psi(1)=1.$ Define a module structure on $V_{\alpha, \beta} $ by, 
\begin{align}
d_n\otimes b. v_k=\psi(b)(\alpha +k +n\beta)v_{k+n},\\
C\otimes b.v_k=0
\end{align}
for all $n,k \in \mathbb Z, b \in B.$
Clearly $V'_{\alpha, \beta}$ are irreducible modules for $Vir_B$, since they are so for $Vir.$ We denote this module by $V'_{\alpha, \beta, \psi}.$
\subsection*{(5)} 
\begin{definition}
A module $V$ for $Vir_B$ is said to be weight module if the action of $d_0$ on $V$ is diagonalizable, i.e $V$ can be decomposed as $V = \displaystyle{\bigoplus_{\lambda \in \mathbb C} V_\lambda}$, where  $V_\lambda=\{ v \in V: d_0v=\lambda v  \}.$ The space $V_\lambda$ is called the weight space with respect to the weight $\lambda$ of $V$ and elements of $V_\lambda$ are called weight vectors of $V$ of weight $\lambda$. Moreover we say that $V$ is a highest weight (lowest weight) module if $Vir^+_B.v=0$ $(Vir^-_B.v=0)$ for some non-zero weight vector $v$ of $V$.
\end{definition}

Let $\phi: Vir^0_B \to \mathbb C$ be a one dimensional representation of $Vir^0_B$. Extend $\mathbb C \phi$ to a module for $ Vir^0_B \oplus Vir^+_B$ by defining action of $Vir^+_B$ as zero. Then define the Verma module,
  $$M(\phi)=U(Vir_B)\bigotimes_{U(Vir^0_B \oplus Vir^+_B)}\mathbb C \phi.  $$  
Clearly $M(\phi)$ is a highest weight module with highest weight $\phi(d_0)$ and $M(\phi)=\displaystyle{\bigoplus_{i\in \mathbb Z^+}}M(\phi)_{\phi(d_0)-i}.$ Note that $\widetilde{v_\phi}=1\otimes 1_\phi$ is the highest weight vector of $M(\phi)$, where $1_\phi$ is the unit in $\mathbb C \phi$. Moreover $M(\phi) \simeq U(Vir^-_B)$ as $U(Vir^-_B)$-modules. \\

Let $N(\phi)$ be the unique maximal proper sub-module of $M(\phi)$. Then
       $$V(\phi)=M(\phi)/N(\phi)  $$
 is the irreducible highest weight module corresponding to $\phi$. It is a highest weight module with highest weight $\phi(d_0)$ and highest weight vector $v_\phi$, is the image of $\widetilde{v_\phi}$ in $V(\phi)$. Moreover $V(\phi)=\displaystyle{\bigoplus_{i\in \mathbb Z^+}}V(\phi)_{\phi(d_0)-i}.$\\
 
Consider the module $V^\phi_{\alpha, \beta, \psi}=     V(\phi)\otimes V'_{\alpha, \beta,\psi}$ for $Vir_B.$ Clearly this a weight module for $Vir_B$ and 
   $$ V^\phi_{\alpha, \beta, \psi}= \displaystyle{\bigoplus_{n\in \mathbb Z}}(V^\phi_{\alpha, \beta, \psi})_{\phi(d_0)+\alpha +n},$$
where $(V^\phi_{\alpha, \beta, \psi})_{\phi(d_0)+\alpha +n}=  \displaystyle{\bigoplus_{i\in \mathbb Z^+}}(V(\phi)_{\phi(d_0)-i}\otimes \mathbb C v_{n+i})$. Hence every weight spaces of $V^\phi_{\alpha, \beta, \psi}$ are infinite dimensional. In our discussion we always take   $0\leq Re \alpha <1$ and $\beta \neq 1$ for all modules  $V^\phi_{\alpha, \beta, \psi}$.\\
\section{Main Results}
\begin{lemma}\label{l3.1}
$ V^\phi_{\alpha, \beta, \psi}$ is generated by $\{ v_\phi \otimes v_m: m \in \mathbb Z \}$ or $\{ v_\phi \otimes v_m: m \in \mathbb Z-\{0\} \}$ according as $(\alpha,\beta)\neq (0,0)$ or $(\alpha, \beta)=(0,0)$, where $v_\phi$ is the highest weight vector of $V(\phi)$.
\end{lemma}
\begin{proof}
Note that $V(\phi)=U(Vir^-_B)v_\phi$. Consequently $ V^\phi_{\alpha, \beta, \psi}$ is spanned by $\{uv_\phi \otimes v_m: u \in U(Vir^-_B)$ and $m \in \mathbb Z$ or $\mathbb Z-\{0\} \}$. Hence the lemma follows from Lemma \ref{l2.1}. 
\end{proof}
\begin{proposition}\label{p3.1}
$End_{Vir_B}( V^\phi_{\alpha, \beta, \psi})\simeq \mathbb C$.
\end{proposition}
\begin{proof}
Let $T \in End_{Vir_B}( V^\phi_{\alpha, \beta, \psi})$. Therefore $T(v_\phi \otimes v_m)$ and $(v_\phi \otimes v_m)$ must have same weight. Let \\
$T(v_\phi \otimes v_m)= \displaystyle{\sum_{i=0}^{i=k}}c_ix_{-i}\otimes v_{m+i}$, where $x_{-i} \in V(\phi)_{\phi(d_0)-i}$, $c_i \in \mathbb C$ for $0\leq i \leq k$ and $x_0=v_\phi.$\\
Let us choose $n (> k)\in \mathbb N$ sufficiently large such that \\
\begin{equation}\label{eq3.1}
(\alpha + m+n\beta)(\alpha +n+m +n \beta) \neq 0 
\end{equation}
\begin{equation}\label{eq3.2}
(\alpha+i+m+2n\beta)(\alpha + m+n\beta)(\alpha +n+m +n \beta)\neq  (\alpha+m +2n\beta)(\alpha+i+m+n\beta)(\alpha+i+n+m+n\beta)
\end{equation}
for $1 \leq i \leq k$.

 This is possible since equality of \ref{eq3.1} and \ref{eq3.2} gives only finitely many solutions for $n$.\\
Set $w= d_{2n} -\frac{(\alpha+m+2n\beta)}{(\alpha+m+n\beta)(\alpha+m+n+n\beta)}d_n^2$, then $w.(v_\phi \otimes v_m)=0$ and $w.(x_{-i}\otimes v_{m+i}) \neq 0$, (due to \ref{eq3.2}) for $1 \leq i \leq k$. Moreover, it is easy to see that $w.(x_{-i}\otimes v_{m+i})$ are linearly independent for $1 \leq i \leq k$. \\
Therefore $T(w.v_\phi \otimes v_m)=w.T(v_\phi \otimes v_m)$ gives $c_i =0$ for $1 \leq i \leq k$. Hence we have 
  $$T(v_\phi \otimes v_m)=c_m(v_\phi\otimes v_m) $$ for some $c_m \in \mathbb C$ and for all $m \in \mathbb Z$ or $\mathbb Z -\{0\}$ according as $(\alpha,\beta)\neq (0,0)$ or $(\alpha,\beta)=(0,0)$.\\ 
  Now consider three cases.\\
 {\bf Case I :} Let us assume $\beta \neq 0.$ Fix two arbitrary integers $m,n.$ Choose $l \in \mathbb N$ in such a way that $\alpha +m + l\beta \neq 0 $ and $\alpha + n + (l+m-n)\beta \neq 0.$ Then, 
            $$T(d_l.v_\phi \otimes v_m)=d_l.T(v_\phi \otimes v_m)$$
implies that $c_{m+l}=c_m. $\\
Again, $T(d_{l+m-n}.v_\phi\otimes v_n)=d_{l+m-n}.T(v_\phi \otimes v_n)$ implies that $c_{m+l}=c_n.$ Hence we have $c_m=c_{m+l}=c_n$.\\
{\bf Case II :} Let us assume $\beta =0, \alpha \neq 0.$ Fix two arbitrary integers $m,n$. Since $0\leq Re \alpha <1$, so $\alpha + m \neq 0$ and $\alpha +n \neq 0$. Choose any $l \in \mathbb N$ and consider $T(d_l.v_\phi \otimes v_m)=d_l.T(v_\phi \otimes v_m)$, $T(d_{l+m-n}.v_\phi\otimes v_n)=d_{l+m-n}.T(v_\phi \otimes v_n)$ which gives us $c_m =c_{m+l}=c_n.$\\
{\bf Case III :} Let us assume $\beta =0, \alpha=0.$ In the same way like Case II we have $c_m=c_n$ for all $m,n \in \mathbb Z-\{0\}.$\\
Combining all the three cases we have $c_m =c $ for some $c \in \mathbb C$ and for all  $m \in \mathbb Z$ or $\mathbb Z -\{0\}$ according as $(\alpha,\beta)\neq (0,0)$ or $(\alpha,\beta)=(0,0)$. Hence we have the result by Lemma \ref{l3.1}.
\end{proof}
From the Proposition \ref{p3.1} we have, modules $V^\phi_{\alpha, \beta , \psi}$ are indecomposable. Now we find out a condition under which these modules will be irreducible. Before that we prove a necessary and sufficient condition for irreducibility of $V^\phi_{\alpha, \beta , \psi}$.
\begin{theorem}\label{t3.1}
For $\alpha+\beta \notin \mathbb Z,$   $V^\phi_{\alpha, \beta , \psi}$ is irreducible if and only if $V^\phi_{\alpha, \beta , \psi}$ is cyclic on every $v_\phi \otimes v_m$ for all $m \in \mathbb Z$.
\end{theorem}
\begin{proof}
Let $V^\phi_{\alpha, \beta , \psi}$ be cyclic on every $v_\phi \otimes v_m$ for all $m \in \mathbb Z$. Let $W$ be any non-zero sub-module of $V^\phi_{\alpha, \beta , \psi}$. Let $w \in W$ be a non-zero weight vector. Let
  $$w= \displaystyle{\sum_{i=0}^{i=n}}x_{-i}\otimes v_{m+i}, $$ 
where $x_{-i} \in V(\phi)_{\phi(d_0)-i}$, for $0\leq i \leq n$.\\
We use induction on $n$ to show that $v_\phi \otimes v_k \in W$ for some $k \in \mathbb Z$. Note that if $n=0,$ then we are done.\\
Assume that $n>0$ and $x_{-n} \neq 0.$ Then there exists a $b \in B$ such that $d_1 \otimes b.x_{-n} \neq 0$ or $d_2 \otimes b .x_{-n} \neq 0$. If both of them are zero for all $b \in B$, then $x_{-n}$ would be a highest weight vector for $V(\phi)$ and hence a scalar multiple of $v_\phi$. Then $w $ cannot be a weight vector. \\
Without loss of generality assume that  $d_1 \otimes b.x_{-n} \neq 0$ for some $b \in B$ (other case can be done in same way).\\
{\bf Case I :} Let us assume $\beta \neq 0$. Choose $l (>> n) \in \mathbb N$ in such a way that $(\alpha+m+n+\beta)\{\alpha+m+n+1+ (l-1)\beta\} \neq 0$. Let
     $$X=d_l \otimes b -\frac{\alpha+m+n+ l\beta}{(\alpha+m+n+\beta)\{\alpha+m+n+1+ (l-1)\beta\}}d_{l-1}d_1\otimes b .$$

By computing the action of $X$ on $w$ we have,\\

$X.w=$\\

$\displaystyle{\sum_{i=0}^{i=n-1}} \psi(b)\lbrace(\alpha+m+i+l\beta)-\frac{(\alpha+m+i+\beta)(\alpha+m+n+l\beta)\{\alpha+m+i+1+(l-1)\beta\}}{(\alpha+m+n+\beta)\{\alpha+m+n+1+ (l-1)\beta\}} \rbrace x_{-i}\otimes v_{m+i+l}$\\
$$-\displaystyle{\sum_{i=0}^{i=n}}\frac{(\alpha+m+n+l\beta)\{\alpha+m+i+(l-1)\beta\}}{(\alpha+m+n+\beta)\{\alpha+m+n+1+ (l-1)\beta\} }d_1\otimes b. x_{-i}\otimes v_{m+i+l-1}$$
First summation of the above expression on $X.w$ runs upto $n-1$ terms because
of $X.v_{m+n}=0.$ Moreover it is also clear that $X.w$ is a weight vector of the form $y_{0}\otimes v_{s_1}+....+y_{-r}\otimes v_{s_k}$ such that $r <n$ and $y_{-i}\in V(\phi)_{\phi(d_0)-i}.$ Therefore to complete the induction process it is sufficient to show that $X.w \neq 0.$\\
Note that the components of $X.w$ on $V(\phi)_{\phi(d_0)-i}\otimes \mathbb C v_{m+l+i} $ is 

$$ \psi(b)\lbrace(\alpha+m+i+l\beta)-\frac{(\alpha+m+i+\beta)(\alpha+m+n+l\beta)\{\alpha+m+i+1+(l-1)\beta\}}{(\alpha+m+n+\beta)\{\alpha+m+n+1+ (l-1)\beta\}} \rbrace x_{-i}\otimes v_{m+i+l}  $$
$$ -\frac{(\alpha+m+n+l\beta)\{\alpha+m+i+1+(l-1)\beta  \}}{(\alpha+m+n+\beta)\{\alpha+m+n+1+ (l-1)\beta\}}d_1\otimes b. x_{-(i+1)}\otimes v_{m+i+l}.$$
If for some $0 \leq i \leq n-1 $, $d_1\otimes b. x_{-(i+1)}$ and $x_{-i}$ are linearly independent then the components of $X.w$ on  $V(\phi)_{\phi(d_0)-i}\otimes \mathbb C v_{m+l+i} $ will be zero, only if the coefficients of $x_{-i}\otimes v_{m+i+l}$ and $d_1\otimes b. x_{-(i+1)}\otimes v_{m+i+l}$ are zero. But this is possible only for finitely many values of $l$. On the other hand if all $d_1\otimes b. x_{-(i+1)}$ and $x_{-i}$ are linearly dependent, then we consider $d_1\otimes b.x_{-n}=k_0x_{-(n-1)},$ for some $k_0 \in \mathbb C$. Now the coefficient of $x_{-(n-1)}\otimes v_{m+l+n-1}$ will be zero only for finitely many values of $l$. Hence in any case we can find $l$ large enough such that $X.w \neq 0$. 
\\
{\bf Case II :} Let us assume $\beta =0$. Choose $l (>>n) \in \mathbb N$. Note that $\alpha + \beta \notin \mathbb Z$ and $\beta =0$ implies that $\alpha \notin \mathbb Z$. Let
$$X=d_{2l}\otimes b -\frac{1}{(\alpha+m+n+1)(\alpha+m+n+l)}d_ld_{l-1}d_1\otimes b. $$ 
Note that $X.v_{m+n}=0$. Using this fact we can obtain the expression for $X.w$ as, \\
   $$X.w=\displaystyle{\sum_{i=0}^{i=n-1}}\psi(b)(\alpha+m+i)\lbrace 1- \frac{(\alpha+m+i+1)(\alpha+m+l+i)}{(\alpha+m+n+1)(\alpha+m+n+l)\rbrace}  \rbrace x_{-i}\otimes v_{m+i+2l} $$
  $$ -  \displaystyle{\sum_{i=0}^{i=n}}\frac{(\alpha+m+i)(\alpha+m+i+l-1)}{(\alpha+m+n+1)(\alpha+m+n+l)}d_1\otimes b.x_{-i}\otimes v_{m+i+2l-1} . $$
Therefore $X.w$ is a weight vector of the form $y_{0}\otimes v_{s_1}+....+y_{-r}\otimes v_{s_k}$ such that $r <n$ and $y_{-i}\in V(\phi)_{\phi(d_0)-i}.$ Now proceed like Case I and prove that we can find $l$ large enough such that $X.w \neq 0$.\\
Combining both the cases we have $v_\phi \otimes v_k \in W$ for some $k \in \mathbb Z$. This completes the proof.

\end{proof}
\begin{corollary}\label{p3c3.1}
Suppose $\alpha \pm \beta \notin \mathbb Z$ and $b \in B$ such that $\psi(b) \neq 0$. Then $V^\phi_{\alpha , \beta, \psi}$ is irreducible if $\phi|_{d_0\otimes <b>}=0,$ where $<b>$ denote the ideal generated by $b$ in $B$.
\end{corollary}
\begin{proof}
We assert that $d_{-1}\otimes b.v_\phi=0$. If not, then \\

$d_1\otimes a.(d_{-1}\otimes b.v_\phi)=[d_1\otimes a,d_{-1}\otimes b].v_\phi=-2d_0\otimes ab.v_\phi=\phi(d_0\otimes ab).v_\phi=0$ for all $a \in B$.\\ 
Moreover $d_2\otimes a.(d_{-1}\otimes b.v_\phi)=0$ for all $a \in B$. Thus $d_{-1}\otimes b.v_\phi$ is a highest weight vector of $V(\phi)$, a contradiction.\\
Now, $$d_{-1}\otimes b.(v_\phi\otimes v_{n+1})=\psi(b)(\alpha +n+1 -\beta)v_\phi \otimes v_n.   $$
This implies that $U(Vir_B)(v_\phi \otimes v_n )\subseteq U(Vir_B)(v_\phi\otimes v_{n+1}) $ for all $n \in \mathbb Z$. On the other hand,
  $$ d_1\otimes b.(v_\phi\otimes v_{n})=\psi(b)(\alpha + n+\beta)(v_\phi\otimes v_{n+1})  $$
  implies that $U(Vir_B)(v_\phi \otimes v_n )\supseteq U(Vir_B)(v_\phi\otimes v_{n+1}) $ for all $n \in \mathbb Z$. Therefore,  $U(Vir_B)(v_\phi \otimes v_m )= U(Vir_B)(v_\phi\otimes v_{n}) $ for all $n,m \in \mathbb Z.$ Now by Theorem \ref{t3.1} and Lemma \ref{l3.1}, $V^\phi_{\alpha, \beta, \psi}$ is irreducible.
\end{proof}
\begin{remark}
In particular, if $B$ is a not a local ring, there exists a non-unit $b \in B$ such that $\psi(b)\neq 0$ and hence in this case $<b> \subsetneq B.$
\end{remark}
\section{Isomorphism Class}
In this section we find out conditions when two modules $V^{\phi_1}_{\alpha_1, \beta_1,\psi_1}$ and $V^{\phi_2}_{\alpha_2, \beta_2,\psi_2}$  will be isomorphic as $Vir_B$ modules. 
\begin{lemma}\label{l34.1}
$V^{\phi_1}_{\alpha_1, \beta_1,\psi_1} \simeq V^{\phi_2}_{\alpha_2, \beta_2,\psi_2}$ imply that $\psi_1=\psi_2$.
\end{lemma}
\begin{proof}
To prove $\psi_1=\psi_2$, it is sufficient to prove that $ker\psi_1=ker \psi_2$. Because then it is easy to see that, $\psi_1=c\psi_2$ for some $c \in \mathbb C$. Hence $\psi_1(1)=\psi_2(1)=1$ gives us $\psi_1=\psi_2$.\\
Let $ker\psi_1 \neq ker\psi_2$ and $b \in ker\psi_1 - ker\psi_2 $. Let $T$ be the isomorphism between $V^{\phi_1}_{\alpha_1, \beta_1,\psi_1}$ and $ V^{\phi_2}_{\alpha_2, \beta_2,\psi_2}$. Then for some $k \in \mathbb N$, $v_{\phi_2} \otimes v_k$ has a pre-image under $T$. Since $T$ is an isomorphism this pre-image is a weight vector. Let\\
 $$ T(\displaystyle{\sum_{i=0}^{i=n}}  c_ix_{-i}\otimes v_{i+p})=v_{\phi_2}\otimes v_k ,$$
 for some $p \in \mathbb Z$ and $x_{-i} \in V(\phi)_{\phi(d_0)-i}$, for $0\leq i \leq n$.\\

{\bf Case I :} Let us assume $\beta_2 \neq 0.$ Choose $l (>>n) \in \mathbb N$ such that $\alpha_2+k+l\beta_2 \neq 0$. Now consider, 
\begin{align}\label{a4.1}
T(d_l\otimes b.\displaystyle{\sum_{i=0}^{i=n}}  c_ix_{-i}\otimes v_{i+p})=d_l\otimes b.v_{\phi_2}\otimes v_k,
\end{align}

which is absurd, since left hand side of \ref{a4.1} is zero whereas right side is non-zero.\\

{\bf Case II :} Let us assume $\beta_2 =0$. Choose $l >>n,$ then right hand side of \ref{a4.1} becomes $\psi_2(b)(\alpha_2+k)v_{\phi_2}\otimes v_{l+k} \neq 0$, since $\alpha_2 +k =0 $ imply that $\alpha_2 \in \mathbb Z, $ so $\alpha_2=0=k$, a contradiction.

\end{proof}

\begin{theorem}\label{p3t4.1}
$V^{\phi_1}_{\alpha_1, \beta_1,\psi_1} \simeq V^{\phi_2}_{\alpha_2, \beta_2,\psi_2}$ if and only if $\psi_1=\psi_2, \phi_1=\phi_2, \alpha_1=\alpha_2, \beta_1=\beta_2.$
\end{theorem}
\begin{proof}
Let $T:V^{\phi_1}_{\alpha_1, \beta_1,\psi_1} \to V^{\phi_2}_{\alpha_2, \beta_2,\psi_2}$ be the isomorphism.\\
{\bf Case I :} Let us assume $(\alpha_1,\beta_1)\neq (0,0)$. Since $T$ is an isomorphism, for all $l \in\mathbb Z$ image of $v_{\phi_1}\otimes v_l$ will be a weight vector of same weight with $v_{\phi_1}\otimes v_l$. Let
\begin{align}\label{p3a4.2}
T(v_{\phi_1}\otimes v_l)= \displaystyle{\sum_{i=0}^{i=k}}  c_ix_{-i}\otimes v_{i+p},
\end{align} 
 for some $p \in \mathbb Z$, $c_i \in \mathbb C$ and $x_{-i} \in V(\phi)_{\phi(d_0)-i}$, for $0\leq i \leq k.$ Since $v_{\phi_1}\otimes v_0$ and its image has same weight, so we have $\phi_1(d_0)+\alpha_1+l=\phi_2(d_0)+\alpha_2+p$.\\
   For all $m,n > k$, $ T(d_m d_n.v_{\phi_1}\otimes v_l)=d_m d_n.T(v_{\phi_1}\otimes v_l)$ implies that,
   
\begin{align}\label{p3a4.3}
(\alpha_1+l+n\beta_1)(\alpha_1+n+l+m\beta_1)T(v_{\phi_1}\otimes v_{m+n+l})=
\end{align}  
\begin{align*}
\displaystyle{\sum_{i=0}^{i=k}}c_i(\alpha_2+p+i+n\beta_2)(\alpha_2+p+i+n+m\beta_2)   x_{-i}\otimes v_{i+p+m+n}.
\end{align*}
Again, for all $m,n >k$, $T(d_{m+n}.v_{\phi_1}\otimes v_l)=d_{m+n}.T(v_{\phi_1}\otimes v_l)$ implies that,

\begin{align}\label{p3a4.4}
\{\alpha_1+l+(m+n)\beta_1\}T(v_{\phi_1}\otimes v_{m+n+l})=\displaystyle{\sum_{i=0}^{i=k}}c_i (\alpha_2+p+i+(m+n)\beta_2)              x_{-i}\otimes v_{i+p+m+n}.
\end{align}
From \ref{p3a4.3} and \ref{p3a4.4} for all $i$ with $c_i \neq 0$ we have, 
\begin{align}\label{p3a4.5}
(\alpha_1+l+n\beta_1)(\alpha_1+l+n+m\beta_1)(\alpha_2+p+i+(m+n)\beta_2)             =
\end{align}
\begin{align*}
\{\alpha_1+l+(m+n)\beta_1\}(\alpha_2+p+i+n\beta_2)(\alpha_2+p+i+n+m\beta_2).
\end{align*}
From \ref{p3a4.5} we obtain a polynomial in $m,n $, which is given by
$$\beta_1\beta_2(\beta_1-\beta_2)mn(m+n) +\{(\alpha_1+l)(\alpha_2+p+i)(\beta_2-\beta_1)+\beta_1(\alpha_2+p+i)^2-\beta_2(\alpha_1+l)^2  \}(m+n)$$
$$+\{(\alpha_1+l)\beta_2(\beta_2-1-2\beta_1)+\beta_1(\alpha_2+p+i)(1+2\beta_2-\beta_1)\}mn+\beta_1\beta_2(\alpha_2+p+i-\alpha_1-l)(m^2+n^2) $$
$$+(\alpha_1+l)(\alpha_2+p+i-\alpha_1-l)(\alpha_2+p+i)=0. $$
Since this equation holds for all $m,n >k$, so we have,
\begin{align}\label{p3a4.6}
\beta_1\beta_2(\beta_1-\beta_2)=0,
\end{align}
\begin{align}\label{p3a4.7}
(\alpha_1+l)(\alpha_2+p+i)(\beta_2-\beta_1)+\beta_1(\alpha_2+p+i)^2-\beta_2(\alpha_1+l)^2=0,
\end{align}
\begin{align}\label{p3a4.8}
(\alpha_1+l)\beta_2(\beta_2-1-2\beta_1)+\beta_1(\alpha_2+p+i)(1+2\beta_2-\beta_1)=0,
\end{align}
\begin{align}\label{p3a4.9}
\beta_1\beta_2(\alpha_2+p+i-\alpha_1-l)=0,
\end{align}
\begin{align}\label{p3a4.10}
(\alpha_1+l)(\alpha_2+p+i-\alpha_1-l)(\alpha_2+p+i)=0.
\end{align}
{\bf Sub-case I :} Let $\beta_1 =0$, then $\alpha_1+l \neq 0$, since $(\alpha_1,\beta_1)\neq (0,0)$ and $0 \leq Re\alpha_1 <1.$  From \ref{p3a4.8}, we have $\beta_2=0$, as $\beta_2 \neq 1$. If $\alpha_2+p+i = 0$, then $0 \leq Re\alpha_2 <1$ implies that $\alpha_2=p+i=0$. Then from \ref{p3a4.2}, we have 
     $$T(v_{\phi_1}\otimes v_l)= c_{-p}x_p\otimes v_0. $$
Then choosing a natural number $n > -p$ we have, 
  $$T(d_n.v_{\phi_1}\otimes v_l)=0 ,$$ which implies that $\alpha_1+l=0$, a contradiction. Therefore $\alpha_2+p+i \neq 0$, Hence from \ref{p3a4.10} and $0 \leq  Re\alpha_1, Re\alpha_2 <1,$ we have $\alpha_1-\alpha_2=p+i-l=0.$\\
{\bf Sub-case II :} Let $\beta_2 = 0$, then from \ref{p3a4.8} we have $\beta_1 =0$ or $\alpha_2 +p +i =0$, as $\beta_1\neq 1$. If $\beta_1 =0 $, then by Sub-case I we have $\alpha_1=\alpha_2$. If $\beta_1 \neq 0$, then $\alpha_2 +p+i=0$ implies that $\alpha_2=p+i=0$. Again from Sub-case I we get that $\alpha_1 \neq 0 $ implies $\alpha_2+p+i \neq 0$, hence $\alpha_1=0$. Moreover we have, 
 $$ T(v_{\phi_1}\otimes v_l)= c_{-p}x_p\otimes v_0.  $$
  Hence for all $n>-p$ we have, 
  $$T(d_n.v_{\phi_1}\otimes v_l)=0 ,$$ 
  this implies that $l+n\beta_1=0$ for all $n>-p$, a contradiction.\\
  {\bf Sub-case III :} Let $\beta_1\beta_2 \neq 0$. Note that from \ref{p3a4.6} and \ref{p3a4.9} we have $\beta_1=\beta_2$, $\alpha_1=\alpha_2$.\\ 
  
Hence in all cases we have $\alpha_1=\alpha_2$, $\beta_1=\beta_2$ and $p+i=l$. Therefore $T$ is given by $  T(v_{\phi_1}\otimes v_l)= c_{l}x_{-(l-p_l)}\otimes v_l,  $ for some $c_{l} \in \mathbb C^{*}$ and $p_l\in \mathbb Z$ such that $l\geq p_l$. From Lemma \ref{l3.1} we have,
   $$\displaystyle{\sum_{l \in \mathbb Z}} U(Vir_B^-)(v_{\phi_1} \otimes v_l)= V^{\phi_1}_{\alpha_1,\beta_1, \psi_1} . $$
   Since $T$ is an isomorphism, we have
   $$ \displaystyle{\sum_{l \in \mathbb Z}} U(Vir_B^-)(x_{-(l-p_l)} \otimes v_l)= V^{\phi_2}_{\alpha_2,\beta_2, \psi_2} .  $$
Last equality holds only when $p_l=l$ for all $l \in \mathbb Z$ . Hence we have  $T(v_{\phi_1}\otimes v_l)= c_{l}'v_{\phi_2}\otimes v_l,  $ for some $c_l' \in \mathbb C^*$ and for all $l \in \mathbb Z$. Now consider the relations,
\begin{align}\label{p3a4.11}
T(d_0\otimes b.v_{\phi_1}\otimes v_l) =d_0\otimes bT(v_{\phi_1}\otimes v_l) 
\end{align} 
\begin{align}\label{p3a4.12}
T(C\otimes b.v_{\phi_1}\otimes v_l) =C\otimes bT(v_{\phi_1}\otimes v_l) 
\end{align}
 for all $b \in B$.
 Then using Lemma \ref{l34.1}, \ref{p3a4.11} and \ref{p3a4.12} we have $\phi_1(d_0\otimes b)=\phi_2(d_0\otimes b)$ and $\phi_1(C\otimes b)=\phi_2(C\otimes b)$ for all $b \in B$.\\
 
 {\bf Case II :} Let $(\alpha_1,\beta_1)=(0,0).$ In this case we proceed similarly like Case I and conclude that  $T(v_{\phi_1}\otimes v_l)= c_{l}'v_{\phi_2}\otimes v_l,  $ for some $c_l' \in \mathbb C^*$ and for all $l \in \mathbb Z-\{0\}$. Hence we have $ \phi_1=\phi_2, \alpha_1=\alpha_2, \beta_1=\beta_2.$\\
Therefore by Case I, Case II and Lemma \ref{l34.1} we have the result.

\end{proof}

\begin{remark}
Corollary \ref{p3c3.1} combined with Theorem \ref{p3t4.1} gives us a collection of non-isomorphic irreducible modules for $Vir_B$ with infinite dimensional weight spaces.
\end{remark}

\end{document}